\documentclass[12pt]{article}

\usepackage{lipsum} 
\usepackage{hyperref} 
\usepackage{authblk} 
\usepackage{geometry} 
\usepackage{graphicx}
\usepackage{multirow}
\usepackage{amsmath,amssymb,amsfonts}
\usepackage{amsthm}
\usepackage{mathrsfs}
\usepackage[title]{appendix}
\usepackage{xcolor}
\usepackage{textcomp}
\usepackage{manyfoot}
\usepackage{booktabs}
\usepackage{algorithm}
\usepackage{algorithmicx}
\usepackage{algpseudocode}
\usepackage{listings}
\usepackage[numbers,sort&compress]{natbib}
\usepackage{physics}
\usepackage{enumitem}
\usepackage{subcaption}
\usepackage{float}

\newcommand{\keywords}[1]{\par\noindent\textbf{Keywords:} #1}

\theoremstyle{plain}
\newtheorem{theorem}{Theorem}[section]

\newtheorem{corollary}[theorem]{Corollary}
\newtheorem{lemma}[theorem]{Lemma}

\theoremstyle{definition}

\newtheorem{remark}{Remark}[section]
\newtheorem{definition}{Definition}[section]

\title{An Operator-Theoretic Framework for the  Optimal Control Problem of Nonlinear Caputo Fractional Systems}
\author[1]{Dev Prakash Jha}
\author[2]{Raju K George}
\affil[1,2]{\small\textit{Mathematics, Indian Institute of Space Science and Technology,}\\ \small\textit{Valiamala, Thiruvananthapuram 695547, Kerala, India}}

\date{\today}

\begin{document}

\maketitle

\begin{abstract}
This paper addresses the optimal control problem for a class of nonlinear fractional systems involving Caputo derivatives and nonlocal initial conditions. The system is reformulated as an abstract Hammerstein-type operator equation, enabling the application of operator-theoretic techniques. Sufficient conditions are established to guarantee the existence of mild solutions and optimal control-state pairs. The analysis covers both convex and non-convex scenarios through various sets of assumptions on the involved operators. An optimality system is derived for quadratic cost functionals using the Gâteaux derivative, and the connection with Pontryagin-type minimum principles is discussed. Illustrative examples demonstrate the effectiveness of the proposed theoretical framework.
\end{abstract}

\keywords{Semi-linear systems, Fractional calculus, Mild solution, Non-autonomous systems,  approximate controllability, Evolution operator, Schauder's fixed-point theorem, Ascoli-Arzelà theorem}

\section{Introduction}
\label{sec:intro}
Fractional differential equations have gained prominence as essential tools for modeling a wide range of phenomena in science and engineering. These equations are widely utilized in fields such as viscoelasticity, electrochemistry, control systems, porous media, and electromagnetics, among others (see \cite{metzler1995relaxation,mainardi1997fractional,hilfer2000applications,gaul1991damping,diethelm1999solution}). Over the past few years, considerable progress has been made in both the theoretical and practical aspects of fractional differential equations, as highlighted in the monographs by Kilbas et al. \cite{kilbas2006theory}, Miller and Ross \cite{miller1993introduction}, and Podlubny \cite{podlubny1998fractional}, along with several research papers \cite{jha2024exact,jha2024existence,sakthivel2011approximate,liu2015approximate,zhou2009existence} and related references.

Throughout this paper, unless explicitly mentioned otherwise, we adopt the following notations. We consider $E$ as a Hilbert space endowed with the norm $\|\cdot\|$. Define $J = [0, a] \subset \mathbb{R}$, and let $C(J, E)$ represent the Banach space of continuous functions mapping $J$ to $E$, equipped with the norm $\|u\|_C = \sup_{t \in J} \|u(t)\|$ for all $u \in C(J, E)$. The primary aim of this paper is to derive sufficient conditions ensuring the approximate controllability of a specific class of abstract fractional evolution equations (FEE) governed by a control function of the form:

\begin{equation}\label{eq:P}\tag{P}
	\begin{cases}
		{}^{C}D^{\alpha} u(t) = Au(t) + Bv(t) + f(t, u(t)), & t \in J, \\
		u(0) = \sum_{k=1}^{m} c_k u(t_k),
	\end{cases}
\end{equation}
The state variable \( u(\cdot) \) takes values in the Hilbert space \( E \), while \( {}^{C}D^\alpha \) represents the Caputo fractional derivative of order \( 0 < \alpha < 1 \). The operator \( A \) serves as the infinitesimal generator of a strongly continuous semigroup \( T(t) \) consisting of bounded operators on \( E \). The control function \( v(\cdot) \) belongs to the Hilbert space \( L^2(J, U) \), where \( U \) is another Hilbert space. Additionally, the operator \( B \) is a bounded linear mapping from \( U \) into \( H := L^2([0,a],E) \). The function \( f : J \times E \to E \) is satisfies the Caratheodory Condition. Finally, for any \( k = 1, 2, 3, \dots, m \), the values \( u(t_k) \) belong to \( E \).

On the other hand, in certain physical applications, employing a nonlocal initial condition has been found to be more effective than the traditional initial condition \( u(0) = u_0 \). For instance, Deng \cite{deng1993exponential} utilized the nonlocal condition in system (\ref{eq:P}) to model the diffusion of a small quantity of gas within a transparent tube in 1993. In this context, the system's condition (\ref{eq:P}) enables additional measurements at specific time points \( t_k \) (\( k = 1, 2, \dots, m \)), thereby providing more accurate data compared to relying solely on measurements taken at \( t = 0 \).

Moreover, Byszewski \cite{byszewski1999existence} demonstrated in 1999 that when \( c_k \neq 0 \) for \( k = 1, 2, \dots, m \), the approach can be applied in kinematics to analyze the trajectory of a physical object, \( t \mapsto u(t) \). This remains valid even if the positions \( u(0), u(t_1), \dots, u(t_m) \) are unknown, as long as the nonlocal condition in system (\ref{eq:P}) is satisfied.

Therefore, when analyzing certain physical phenomena, a nonlocal condition may provide a more accurate representation than the conventional initial condition \( u(0) = u_0 \). The significance of nonlocal conditions has been extensively studied in \cite{jha2024exact,ezzinbi2007existence,liang2015controllability,wang2017approximate,byszewski1991theorems,boucherif2009semilinear}.


\section{Preliminaries}

\label{sec:prel}
In this section, we present essential notations, definitions, and fundamental concepts that will be utilized throughout this paper. Let $E$ and $U$ be Hilbert spaces equipped with the norms $\lVert \cdot \rVert$ and $\lVert \cdot \rVert_{U}$, respectively. The space of all continuous functions mapping the interval $J$ into $E$ is denoted by $C(J,E)$, which forms a Hilbert space when endowed with the supremum norm:
\begin{equation}\label{eq:pre_1}
	\lVert u\rVert_C= \sup_{t \in J} \lVert u(t)\rVert, \quad u \in C(J,E).
\end{equation}

The Banach space $\mathcal{L}(E)$ represents the collection of all bounded linear operators on $E$, equipped with the operator norm topology. Furthermore, let $L^2(J,U)$ denote the Banach space consisting of all $U$-valued Bochner square-integrable functions defined on $J$, endowed with the norm:

\begin{equation}\label{eq:pre_2}
    \|u\|_2 = \left( \int_0^a \|u(t)\|_U^2 dt \right)^{\frac{1}{2}}, \quad u \in L^2(J,U).
\end{equation}

\begin{definition}\label{def:preli_1}
    (See \cite{kilbas2006theory}) The fractional integral of order $\alpha$ with a lower limit at $0$ for a function $f$ is expressed as
    \[
    I^\alpha f(t) = \frac{1}{\Gamma(\alpha)} \int_0^t \frac{f(s)}{(t-s)^{1-\alpha}} \, ds, \quad t > 0, \quad \alpha > 0,
    \]
    given that the integral on the right-hand side is well-defined for all $t \geq 0$. Here, $\Gamma(\alpha)$ represents the Gamma function.
\end{definition}

\begin{definition}\label{def:preli_2}
    (See \cite{kilbas2006theory}) The Caputo fractional derivative of order $\alpha > 0$ with a lower limit at $0$ for a function $f$ is formulated as
    \[
    {}^C D_t^\alpha f(t) = \frac{1}{\Gamma(1-\alpha)} \int_0^t (t-s)^{-\alpha} f'(s) \, ds, \quad t > 0,
    \]
    where $f(t)$ is absolutely continuous.
\end{definition}

\begin{definition}\label{Def:2}
    The Riemann–Liouville fractional derivative of order $\alpha$ with a lower limit at $0$ for a function $f : [0, \infty) \to \mathbb{R}$ is expressed as
\[
{}^{L}D^\alpha f(t) = \frac{1}{\Gamma(n-\alpha)} \frac{d^n}{dt^n} \int_0^t \frac{f(s)}{(t-s)^{\alpha-n+1}} \, ds, \quad t > 0, \, n-1 < \alpha < n.
\]
\end{definition}

\begin{definition}\label{Def:3}
    The Caputo fractional derivative of order $\alpha$ for a function $f : [0, \infty) \to \mathbb{R}$ is given by
\[
{}^{C}D^\alpha f(t) = {}^{L}D^\alpha \left(f(t) - \sum_{k=0}^{n-1} \frac{t^k}{k!} f^{(k)}(0)\right), \quad t > 0, \, n-1 < \alpha < n.
\]
\end{definition}

\begin{remark}
    \begin{enumerate}
        \item If $f \in C^n[0, \infty)$, then
        \[
        {}^{C}D^\alpha f(t) = \frac{1}{\Gamma(n-\alpha)} \int_0^t \frac{f^{(n)}(s)}{(t-s)^{\alpha-n+1}} \, ds = I^{n-\alpha} f^{(n)}(t), \quad t > 0, \quad n-1 < \alpha < n.
        \]
        \item The Caputo derivative of a constant function is always zero.
        \item If $f$ is a function taking values in a Hilbert space $E$, then the integrals in Definitions (\ref{Def:2}) and (\ref{Def:3}) are interpreted in the Bochner sense.
    \end{enumerate}
\end{remark}

If $f$ is a function mapping into $E$, then the integrals in Definitions (\ref{def:preli_1}) and (\ref{def:preli_2}) are considered in the Bochner sense. Specifically, a function $f: [0, +\infty) \to E$ is Bochner integrable if $\|f\|$ is Lebesgue integrable.

\begin{definition}\label{def:Type_M}
    Let $E$ be a real Hilbert space. Define $\mathcal{M}$ as the collection of all operators $N: E \to E$ that satisfy the following condition: for any $N \in \mathcal{M}$,
    \[
    \langle N x_1 - N x_2, x_1 - x_2 \rangle \geq \aleph_1 \|x_1 - x_2\|^2,
    \]
    for all $x_1, x_2 \in E$, where $\aleph_1$ is a given positive constant.
    
    For $N \in \mathcal{M}$, define the parameter
    \[
    \mu(N) = \inf_{\substack{x_1, x_2 \in E \\ x_1 \neq x_2}} \frac{\langle N x_1 - N x_2, x_1 - x_2 \rangle}{\|x_1 - x_2\|^2}.
    \]
\end{definition}

\begin{definition}\label{def:preli_3}
    (See \cite[Definition 2.3]{bazhlekova2001fractional}) A function \( T_\alpha : \mathbb{R}^+ \to \mathcal{L}(E) \) is referred to as an \(\alpha\)-order solution operator generated by \( A \) if it satisfies the following conditions:
    \begin{itemize}
        \item[(i)] The function \( T_\alpha(t) \) is strongly continuous for \( t \geq 0 \), with \( T_\alpha(0) = I \);
        \item[(ii)] For all \( t \geq 0 \), we have \( T_\alpha(t)D(A) \subseteq D(A) \), and the relation \( AT_\alpha(t)x = T_\alpha(t)Ax \) holds for every \( x \in E \);
        \item[(iii)] Given any \( x \in D(A) \) and \( t \geq 0 \), the function \( T_\alpha(t)x \) satisfies the following integral equation:
        \[
        u(t) = x + \frac{1}{\Gamma(\alpha)} \int_0^t (t-s)^{\alpha-1} A u(s) \, ds.
        \]
    \end{itemize}
\end{definition}

\begin{definition}\label{def:preli_4}
    (See \cite[Definition 2.13]{bazhlekova2001fractional}) Let \( 0 < \theta_0 \leq \frac{\pi}{2} \) and \( \omega_0 \in \mathbb{R} \). An \(\alpha\)-order solution operator \( \{T_\alpha(t)\}_{t \geq 0} \) is said to be analytic if it possesses an analytic extension to the sector \( \Sigma_{\theta_0} := \{z \in \mathbb{C} \setminus \{0\} : |\arg z| < \theta_0\} \) and this extension remains strongly continuous on \( \Sigma_\theta \) for each \( \theta \in (0, \theta_0) \).

An analytic \(\alpha\)-order solution operator \(T_\alpha(z)\), defined for \( z \in \Sigma_\theta \) and generated by \( A \), is classified as being of analyticity type \( (\omega_0, \theta_0) \) if, for every \( \theta \in (0, \theta_0) \) and \( \omega > \omega_0 \), there exists a constant \( M = M(\omega, \theta) > 0 \) such that
\[
\|T_\alpha(z)\| \leq M e^{\omega \Re z}, \quad z \in \Sigma_\theta.
\]  
If \( A \) generates an analytic \( \alpha \)-order solution operator \( T_\alpha \) of analyticity type \( (\omega_0, \theta_0) \), we denote this as \( A \in \mathcal{A}^\alpha(\omega_0, \theta_0) \).
\end{definition}

\begin{definition}\label{def:preli_5}
    Let $A$ be a closed linear operator with domain $D(A)$ defined on a Banach space $E$, and let $\alpha > 0$. The resolvent set of $A$ is denoted by $
ho(A)$. The operator $A$ is called the generator of an $\alpha$-order resolvent family if there exist constants $\omega \geq 0$ and a strongly continuous mapping $S_{\alpha} : \mathbb{R}^{+} \to \mathcal{L}(E)$ such that  
    \[
    \{ \nu \in \mathbb{C} : \text{Re} \, \nu > \omega \} \subset \rho(A),
    \]
    and 
    \[
    (\nu^{\alpha}I - A)^{-1}u = \int_{0}^{\infty} e^{-\nu t} S_{\alpha}(t)u \, dt, \quad \text{for } \text{Re} \, \nu > \omega, \quad u \in E.
    \]
    In this context, the family $\{S_{\alpha}(t)\}_{t \geq 0}$ is referred to as the $\alpha$-order resolvent family generated by $A$.
\end{definition}

\begin{definition}\label{def:preli_6}
    An $\alpha$-order solution family $\{T_{\alpha}(t)\}_{t \geq 0}$ is termed compact if, for every $t > 0$, the operator $T_{\alpha}(t)$ is compact. Likewise, an $\alpha$-order resolvent family $\{S_{\alpha}(t)\}_{t \geq 0}$ is called compact if, for every $t > 0$, the operator $S_{\alpha}(t)$ is compact.
\end{definition}

By employing a method analogous to the one utilized in (\cite[Lemma 10]{fan2014characterization}), we establish the continuity of the $
\alpha$-order solution operator $\{T_{\alpha}(t)\}_{t \geq 0}$ and the $
\alpha$-order resolvent operator $\{S_{\alpha}(t)\}_{t \geq 0}$ in the uniform operator topology.

\begin{lemma}\label{lemma:1}
    Suppose that $A \in \mathcal{A}^{\alpha}(\omega_{0}, \theta_{0})$, and that both the $
\alpha$-order solution operator $\{T_{\alpha}(t)\}_{t \geq 0}$ and the $
\alpha$-order resolvent operator $\{S_{\alpha}(t)\}_{t \geq 0}$ are compact. Then, for every $t > 0$, the following properties hold:
    \begin{enumerate}
        \item $\lim_{h \rightarrow 0} \|T_{\alpha}(t + h) - T_{\alpha}(t)\| = 0,$ and $\lim_{h \rightarrow 0} \|S_{\alpha}(t + h) - S_{\alpha}(t)\| = 0$;
        \item $\lim_{h \rightarrow 0^{+}} \|T_{\alpha}(t + h) - T_{\alpha}(h)T_{\alpha}(t)\| = 0,$ and $\lim_{h \rightarrow 0^{+}} \|S_{\alpha}(t + h) - S_{\alpha}(h)S_{\alpha}(t)\| = 0$;
        \item $\lim_{h \rightarrow 0^{+}} \|T_{\alpha}(t) - T_{\alpha}(h)T_{\alpha}(t - h)\| = 0,$ and $\lim_{h \rightarrow 0^{+}} \|S_{\alpha}(t) - S_{\alpha}(h)S_{\alpha}(t - h)\| = 0$.
    \end{enumerate}
\end{lemma}

\begin{lemma}
    \label{lemma:2} (See \cite[2.26]{bazhlekova2001fractional}, \cite[Theorem 2.3]{shi2016study})
    Suppose that $A \in \mathcal{A}^{\alpha}(\omega_0, \theta_0)$. Then, for any $t > 0$ and $\omega > \omega_0$, the following inequalities hold:
    
    \begin{equation}\label{eq:MT_1}
        \|T_{\alpha}(t)\|_{\mathcal{L}(E)} \leq Me^{\omega t}, \quad \text{and} \quad \|S_{\alpha}(t)\|_{\mathcal{L}(E)} \leq Ce^{\omega t}(1 + t^{1-\alpha}).
    \end{equation}
    
    Moreover, defining
    \begin{equation}\label{eq:MT_2}
        M_T := \sup_{t \in \mathbb{R}^+} \|T_{\alpha}(t)\|_{\mathcal{L}(E)}, \quad M_S := \sup_{t \in \mathbb{R}^+} Ce^{\omega t}(1 + t^{1-\alpha}),
    \end{equation}
    we derive
    \begin{equation}\label{eq:MT_3}
        \|T_{\alpha}(t)\|_{\mathcal{L}(E)} \leq M_T, \quad \|S_{\alpha}(t)\|_{\mathcal{L}(E)} \leq e^{\omega t}M_S.
    \end{equation}
    
    For additional definitions and properties of the $\alpha$-order solution and resolvent operators, refer to \cite{bazhlekova2001fractional}, \cite{li2012cauchy}, and \cite{lian2017approximate}.
\end{lemma}

\textbf{Assumption(H1):} $\quad \sum_{k=1}^m |c_k| < \frac{1}{M_T}.$

We now restate assumption (H1) in the following manner:
\begin{equation}\label{eq:modified_1}
  \left\| \sum_{k=1}^m c_k T_\alpha(t_k) \right\| \leq M_T \sum_{k=1}^m |c_k| < 1.
\end{equation}

Utilizing equation~(\ref{eq:modified_1}) along with the spectral theorem for operators, we deduce that
\begin{equation}\label{eq:modified_2}
    \mathcal{O} := \left( I - \sum_{k=1}^m c_k T_\alpha(t_k) \right)^{-1}
\end{equation}
exists, is bounded, and satisfies \( D(\mathcal{O}) = E \). Moreover, applying the Neumann series expansion, \( \mathcal{O} \) can be represented as
\begin{equation}\label{eq:modified_3}
    \mathcal{O} = \sum_{n=0}^\infty \left( \sum_{k=1}^m c_k T_\alpha(t_k) \right)^n.
\end{equation}

Therefore,
\begin{equation}\label{eq:chen_4}
    \|\mathcal{O}\| \leq \sum_{n=0}^\infty \left\| \sum_{k=1}^m c_k T_\alpha(t_k) \right\|^n \leq \frac{1}{1 - M_T \sum_{k=1}^m |c_k|}.
\end{equation}

Considering the preceding discussion, as well as results from \cite[Proposition 1.2]{pruss2012evolutionary} and \cite{shi2016study}, it follows that the mild solution to the fractional evolution equation \eqref{eq:P} with initial condition \( u(0) \) is expressed as
\begin{equation}\label{eq:chen_5}
    u(t) = T_\alpha(t) u(0) + \int_0^t S_\alpha(t - s) \big[ B v(s) + f(s, u(s)) \big] ds, \quad t \in J.
\end{equation}
From equation \eqref{eq:chen_5}, we derive the expression
\begin{equation}\label{eq:chen_5_1}
    u(t_k) = T_\alpha(t_k) u(0) + \int_0^{t_k} S_\alpha(t_k - s) \big[ B v(s) + f(s, u(s)) \big] ds, \quad k=1,2,\dots,m.
\end{equation}
Utilizing equations \eqref{eq:P} and \eqref{eq:chen_5_1}, we obtain
\begin{equation}\label{eq:chen_6}
    u(0) = \sum_{k=1}^m c_k T_\alpha(t_k) u(0) + \sum_{k=1}^m c_k \int_0^{t_k} S_\alpha(t_k - s) [Bv(s) + f(s, u(s))] \, ds. 
\end{equation}

Utilizing assumption (H1) along with the definition of the operator $\mathcal{O}$, we derive the following expression:
\begin{equation}\label{eq:chen_7}
    u(0) = \sum_{k=1}^m c_k \mathcal{O} \int_0^{t_k} S_\alpha(t_k - s) \big[Bv(s) + f(s, u(s))\big] \, ds. 
\end{equation}

As a result, combining equations (\ref{eq:chen_5}) and (\ref{eq:chen_7}), we obtain

\begin{align}\label{eq:chen_8}
    u(t) &= \sum_{k=1}^m c_k T_\alpha(t) \mathcal{O} \int_0^{t_k} S_\alpha(t_k - s) \big[Bv(s) + f(s, u(s))\big] \, ds 
\nonumber \\
    &\quad + \int_0^t S_\alpha(t - s) \big[Bv(s) + f(s, u(s))\big] \, ds, \quad t \in J.
\end{align}  

For simplicity, we define the Green's function $G(t, s)$ as:
\begin{equation}\label{eq:chen_9}
    G(t, s) = \sum_{k=1}^m \chi_{t_k}(s) T_\alpha(t) \mathcal{O} S_\alpha(t_k - s) + \chi_t(s) S_\alpha(t - s), \quad t, s \in [0, a],
\end{equation}  
where,

\begin{equation}\label{eq:chen_10}
\chi_{t_k}(s) =
\begin{cases}
c_k, & s \in [0, t_k), \\
0, & s \in [t_k, a],
\end{cases}
\quad
\chi_t(s) =
\begin{cases}
1, & s \in [0, t), \\
0, & s \in [t, a].
\end{cases} 
\end{equation}  

Thus, for given $v \in L^2(J,U)$ and the  based on equations \eqref{eq:chen_8}, \eqref{eq:chen_9}, and \eqref{eq:chen_10}, it follows that the mild solution of the fractional evolution equation (FEE) \eqref{eq:P} is given by (See,\cite{chen2020existence} and \cite{jha2025approximate}  ) 
\begin{equation}\label{eq:chen_11}
    u(t) = \int_0^a G(t, s) \big[Bv(s) + f(s, u(s))\big] \, ds, \quad t \in J.
\end{equation}

Let the cost functional associated with the system (\ref{eq:P}) be given by
\begin{equation*}
    J(v) = \phi(v, u) = \int_0^a h(\tau, v(\tau), u(\tau))\, d\tau,
\end{equation*}
where $h$ is a mapping from $[0,a] \times U \times E \rightarrow \bar{\mathbb{R}}_+$, satisfying the Carathéodory conditions with respect to $t$, $v$, and $x$.

let $H= L^2(J,E)$ and $W= AC(J,U).$ Define the operators $K: H^* \to H,  N: H \to H^*$ and $H: W \to H$ as

\begin{equation}\label{eq:R}\tag{R}
	\begin{cases}
		(Ku)(t) = \int_{0}^{t} G(t,s)u(s) \, ds,\\
         (N u)(t) = f(t, u(t))\\
         (Hv)(t)=\int_{0}^{t} G(t,s)B v(s) \, ds,
	\end{cases}
\end{equation}

By using the above definitions it is easy to see that the equation (\ref{eq:chen_11}) can be reduced
in to the abstract operator equation of Hammerstein type

\begin{equation}\label{eq:Dev_1}
    u=KNu+ Hv.
\end{equation}

\section{Analysis of the Existence of an Optimal Pair in the Abstract Framework}

In this section, we find sufficient conditions for the operators to ensure the existence of an optimal pair for the abstract system (\ref{eq:Dev_1}). Throughout this section, we assume that $X$ and $W$ are reflexive Banach spaces, and let the operator $T : U \rightarrow X$ be such that $T u = x$, where $x \in X$ is a solution operator that has a unique solution $x \in X$ for each control $u \in U$ satisfying (\ref{eq:Dev_1}), that is $x = Tu$. To address this problem, we give the set of sufficient conditions for the operators $L, N$ $K$ that will be used in the main results of this section.

\textbf{Assumptions [$\mathfrak{A}$]:}
\begin{itemize}
    \item[($\mathfrak{A_1}$)]: The operator \( K : X^* \to X \) is linear and compact, and there exists a constant \( d > 0 \) such that \( \langle Kx, x \rangle \geq d \|Kx\|^2 \) for all \( x \in X^* \).
        \item[($\mathfrak{A_2}$)]:  The operator  \( N : X \to X^* \) is continuous, bounded, and negatively monotone.
        \item[($\mathfrak{A_3}$)]:  The operator  \( H : U \to X \) is completely continuous.
\end{itemize}

\textbf{Assumptions [$\mathfrak{B}$]:}
\begin{itemize}
    \item[($\mathfrak{B_1}$)]:  \( K \) is linear and belongs to \( M \).
        \item[($\mathfrak{B_2}$)]: \( N \in \text{Lip} \) and \( \mu(-N) > 0 \) with \( (\mu(K) + \mu(-N))\|N\|^{*{-2}} > 0 \).
        \item[($\mathfrak{B_3}$)]:  \( H \) is completely continuous.
\end{itemize}

\textbf{Assumptions [$\mathfrak{C}$]:}
\begin{itemize}
    \item[($\mathfrak{C_1}$)]:  \( K \) and \( N \) satisfy either ($\mathfrak{A_1}$) and($\mathfrak{A_2}$) or ($\mathfrak{B_1}$) and ($\mathfrak{B_2}$).
        \item[($\mathfrak{C_2}$)]: \( H \) and \( N \) are weakly continuous.
\end{itemize}

\textbf{Assumptions [$\mathfrak{D}$]:}
\begin{itemize}
    \item[($\mathfrak{D_1}$)]:  \( K \) is a bounded linear operator, and there exists a constant \( d > 0 \) such that \( \langle Kx, x \rangle \geq d \|Kx\|^2 \) for all \( x \in X^* \).
        \item[($\mathfrak{D_2}$)]: \( N \) is a continuous bounded operator of type (M).
         \item[($\mathfrak{D_2}$)]:  \( H \) is completely continuous.
\end{itemize}

\begin{lemma}\label{lemma:dev_1}
    Suppose that the operators \( K, N \) and \( H \) satisfy either Assumptions [$\mathfrak{A}$] or Assumptions [$\mathfrak{B}$], then the system operator \( T \) is well defined and is completely continuous.
\end{lemma}

\begin{proof}
    Let Assumptions [$\mathfrak{A}$] hold. Then by Theorem 1 of Hess \cite{hess1971nonlinear} it follows that it \((I - KN)^{-1}\) is well defined, and so is \( T = [I - KN]^{-1}H \). The boundedness and continuity of \([I - KN]^{-1}\) follow from a similar argument given in the proof of Lemma 5.3 of Jha \cite{jha2025approximate}. This, together with complete continuity of \( H \) imply that \( T \) is completely continuous.\\
    If Assumptions [$\mathfrak{B}$] hold, we get the continuity and boundedness of \([I - KN]^{-1}\) by Theorem 2.1 of Dolezal \cite{dolevzal1979monotone} and hence the complete continuity of \([I - KN]^{-1}H \).
    
\end{proof}

Similarly, we have the following result regarding the weak continuity of \( T \).
\begin{lemma}\label{lemma:dev_2}
    Under any one of the set of Assumptions [$\mathfrak{C}$] or Assumptions [$\mathfrak{D}$], the system operator \( T \) is weakly continuous.
\end{lemma}

\textbf{Assumptions [I]:}
\begin{itemize}
    \item[(a)]: The operators \( K, N \) and \( H \) satisfy any one of the sets of Assumptions [$\mathfrak{A}$] or Assumptions [$\mathfrak{B}$].
    \item[(b)]: For a fixed \( x, u \to \phi(u,x) \) is convex and continuous and \( x \to \phi(u,x) \) is continuous and it is uniform for all \( u \in U \).
    \item[(c)]: The control set \( U \) is weakly compact.
\end{itemize}

\begin{theorem}\label{theorem:Th_1_dev}
    Under Assumptions [I], the system (\ref{eq:Dev_1}) has an optimal pair $(u^*, x^*)$.
\end{theorem}
\begin{proof}
We first show that $J$ is weakly lower semicontinuous. That is, $J(u^*) \leq \lim \inf_n J(u_n)$ whenever $u_n \rightharpoonup u^*$ in $U$.\\

Let $g(u,v)$ be real valued function on $U \times X$ defined by\\

\[
g(u,v) = \phi(u,Tv)
\]
Given Assumption [I(b)], $u \to g(u, Tv)$ is continuous and convex for a fixed $v \in X$. Also, Assumption [I(a)] implies (by Lemma \ref{lemma:dev_1}) that it $T$ is completely continuous. This, together with Assumption [I(b)], gives that it $v \to g(u, v)$ is weakly continuous uniformly for all $u \in U$. As all assumptions of Theorem 1 of Browder \cite{browder1965remarks} are satisfied, it follows that
\[
J(u) = \phi(u, Tu) = g(u,v)
\]

is weakly lower semicontinuous.

As $U$ is weakly compact and $J$ is weakly lower semicontinuous, then there exists $u^* \in U$ such that
\[
\phi(u^*, x^*) = J(u^*) \leq \inf_{u \in U} J(u) = \inf_{u \in U} \phi(u,x)
\]
where $x = Tu$ and $x^* = Tu^*$.\\
This proves that it $(u^*, x^*)$ is an optimal pair for the abstract system (\ref{eq:Dev_1}).
\end{proof}

\textbf{Assumptions [II]:}
\begin{itemize}
    \item[(a)]:  The operators $K, N$, and $H$ satisfy either Assumptions [$\mathfrak{C}$] or Assumptions [$\mathfrak{D}$].
    \item[(b)]:  For a fixed $x$, $u \to \phi(u,x)$ is convex and continuous and $x \to \phi(u,x)$ is weakly continuous and is uniform for all $u \in U$.
    \item[(c)]: The control set \( U \) is weakly compact.
\end{itemize}

\begin{theorem}\label{theorem:Th_2_dev}
    Under Assumptions [II], the system (\ref{eq:Dev_1}) has an optimal pair $(u^*, x^*)$.
\end{theorem}

    In most of the optimal control problems, the functional $\phi(u,x)$ is of quadratic type with respect to both $u$ and $x$ and, in such cases, weaker forms of continuity, viz., lower semicontinuity or weak lower semicontinuity of $\phi$ is more accessible to check. So, in the following theorems, we prove the existence of optimal pair $(u^*, x^*)$ with lower semi-continuity (weak lower semicontinuity) assumptions on $\phi$ instead of continuity (weak continuity) assumptions. For such functionals, we investigate directly the unconstrained problem, where $U$ is the whole space $Z$, which is assumed to be a reflexive Banach space.

\textbf{Assumptions [III]:}
\begin{itemize}
    \item[(a)]: The operators $K, N$, and $H$ satisfy either Assumptions [$\mathfrak{A}$] or Assumptions [$\mathfrak{B}$].
    \item[(b)]:  $\phi : Z \times T(Z) \to \mathbb{R}$ is a lower semicontinuous function with respect to the weak topology in $Z$ and norm topology in $T(Z)$ and $\|\cdot\| \to \infty$ implies $\phi(u, Tu) \to \infty, u \in U$.
\end{itemize}

\begin{theorem}\label{theorem:Th_3_dev}
    Under Assumptions [III], the system (\ref{eq:Dev_1}) has an optimal pair $(u^*, x^*)$.
\end{theorem}
\begin{proof}
     Let $\{u_n\}$ be a minimizing sequence of controls in $Z$. That is,
\[
\lim_{n \to \infty} \phi(u_n, Tu_n) = \inf_{u \in Z} \phi(u, Tu) = m (\text{say}).
\]
From Assumption [III(b)], $\{u_n\}$ is bounded. Since $Z$ is reflexive, by extracting a subsequence, we can assume that $u_n \to u^*$ in $Z$. In view of Assumption [III(a)], Lemma \ref{lemma:dev_1} implies that the system operator is completely continuous and hence $Tu_n \to Tu^*$ strongly in $X$. Assumption [III(b)] gives
\[
\phi(u^*, Tu^*) \leq \lim_{n \to \infty} \phi(u_n, Tu_n) \quad \text{whenever } u_n \to u^*.
\]
This in turn implies that
\[
m = \inf_{u \in Z} \phi(u, Tu) \geq \phi(u^*, Tu^*).
\]
That is, $m = \phi(u^*, Tu^*)$ and $(u^*, x^*)$ is the desired optimal pair, where $x^* = Tu^*$.
\end{proof}

When the system operator $T$ is weakly continuous we have the following result, the proof of which follows by using Lemma \ref{lemma:dev_2}.

\textbf{Assumptions [IV]:}
\begin{itemize}
    \item[(a)] The operators $K, N$, and $H$ satisfy either Assumptions [$\mathfrak{C}$] or Assumptions [$\mathfrak{D}$].
    \item[(b)] $\phi : Z \times T(Z) \to \mathbb{R}$ is lower semicontinuous with respect to the weak topologies in $Z$ and $T(Z)$ and further $\|u\| \to \infty$ implies $\phi(u, Tu) \to \infty, u \in Z$.
\end{itemize}

\begin{theorem}\label{theorem:Th_4_dev}
    Under Assumptions [IV] the nonlinear system (\ref{eq:Dev_1}) has an optimal pair $(u^*, x^*)$.
\end{theorem}

Let the explicit representation for the cost functional be given by
\begin{equation}\label{eq:Dev_2}
\phi(u, x) = (u, Ru) + (x, Wx)
\end{equation}
where $R : Z \to Z^*$ is a bounded linear symmetric, strictly monotone and coercive operator and $W : X \to X^*$ is a bounded linear symmetric monotone operator.\\
 As a corollary of the above Theorem  \ref{theorem:Th_4_dev} we have the following result.

\begin{corollary}\label{cor:dev_1}
     If the system operator is weakly continuous then the system (\ref{eq:Dev_1}) with respect to the cost functional (\ref{eq:Dev_2}) has a unique optimal pair $(u^*, x^*)$. 
\end{corollary}

\begin{proof}
    Set $\phi_1(u) = (u, Ru)$ and $\phi_2(x) = (x, Wx)$. So $\phi(u, x)$ can be written as
\[
\phi(u, x) = \phi_1(u) + \phi_2(x).
\]

We first observe that $\phi_1$ and $\phi_2$ are differentiable functionals with gradients $2R$ and $2W$, respectively. Also, $R$ and $W$ are monotone operators by assumption and hence by Theorem 6.1 and Theorem 8.4 of Vainberg \cite{vainberg1973variational} it follows that $\phi_1$ and $\phi_2$ are weakly lower semicontinuous convex functionals. Moreover, the strict monotonicity of $R$ implies that $\phi_1$ is strictly convex and so is $\phi$.

Further, since $R$ is coercive, that is
\[
\frac{(u, Ru)}{\|u\|} \to \infty \quad \text{as } \|u\| \to \infty,
\]
we have that $\phi(u, Tu) \to \infty$ as $\|u\| \to \infty$.

Thus $\phi : Z \times T(Z) \to \mathbb{R}$ is lower semicontinuous in the weak topologies of $Z$ and $T(Z)$ and is coercive. Now applying Theorem \ref{theorem:Th_4_dev} it follows immediately that there exists an optimal pair $(u^*, x^*)$ for the system (\ref{eq:P}) with respect to the cost functional (\ref{eq:Dev_2}). 

The uniqueness of $(u^*, x^*)$ follows from the fact that the cost functional is strictly convex in $Z$.
\end{proof}

\section{OPTIMALITY SYSTEM FOR THE ABSTRACT PROBLEM}

In this section, we investigate Problem 2.2(XXXXXXXX). We assume that the state space $X$ and the control space $Z$ are real Hilbert spaces. Also, throughout this section, the cost functional under consideration is of the form (4.1). That is, $J(u) = \Phi(u, x)$ is of the form
\begin{equation}\label{eq:OPTIMALITY_1}
J(u) = (u, Ru) + (x, Wx)
\end{equation}
where $R$ and $W$ satisfy the earlier assumptions of Section 4(XXXXXXXXXXX). Note that in view of monotonicity and coercivity assumptions on $R$, $R$ is invertible as a bounded linear operator (refer to Joshi and Bose \cite{joshi1985some}). This fact will be used in the subsequent analysis.

Recall that the system operator $T$ is of the form
\begin{equation}\label{eq:OPTIMALITY_2}
T u = [I - K N]^{-1} H u, \quad u \in Z.
\end{equation}

The following lemma gives the existence of the derivative of the system operator $T$ under certain conditions on $K, N$, and $H$.

\textbf{Assumptions [V]:}
\begin{itemize}
    \item[(a)]:  $K, N$, and $H$ are Fréchet differentiable with $K'(x) = K$, $N'(x) = G(x)$ for all $x \in X$ and $H'(u) = H$ for all $u \in Z$. 
    \item[(b)]: $[I - K G(x)]^{-1}$ exists as a bounded linear operator for all $x = Tu, u \in Z$.
\end{itemize}

\begin{lemma}\label{lemma:dev_3}
    Under Assumptions [V], the system operator $T$ is Fréchet differentiable with derivative given by
\begin{equation}\label{eq:OPTIMALITY_3}
T'(u) = [I - K G(x)]^{-1} H, \quad u \in Z \quad \text{and} \quad x = Tu.
\end{equation}
\end{lemma}
\begin{lemma}\label{lemma:dev_4}
  Suppose that the system operator \(T\) is Fréchet differentiable then the cost functional \(J\) given by (\ref{eq:OPTIMALITY_1}) is Fréchet differentiable with derivative \(J'(u)\) given by
\[
\frac{1}{2} J'(u) = \langle T'(u)h, WTu \rangle + \langle h, Ru \rangle, \quad u \in Z.
\]
\end{lemma}

\begin{proof}
     We have \(J(u) = \langle u, Ru \rangle + \langle x, Wx \rangle\).

As \(R\) and \(W\) are bounded symmetric linear operators, we get
\begin{align*}
    & J(u + h) - J(u) - 2\langle T'(u)h, WTu \rangle - 2\langle h, Ru \rangle\\
    & = \langle T(u + h), WT(u + h) \rangle - \langle Tu, WTu \rangle + \langle u + h, R(u + h) \rangle\\
    &\quad - 2\langle h, Ru \rangle - \langle u, Ru \rangle - 2\langle T'(u)h, WTu \rangle\\
    &=2\langle T(u + h) - Tu - T'(u)h, WTu \rangle\\
    &\quad + \langle T(u + h) - Tu, W(T(u + h) - T(u)) \rangle + \langle h, Rh \rangle.
\end{align*}

This implies that
\begin{align}\label{eq:OPTIMALITY_4}
   & \frac{\| J(u + h) - J(u) - 2\langle T'(u)h, WTu \rangle - 2\langle h, Ru \rangle \|}{\| h \|}\notag\\
    & \leq 2\|WTu\| \frac{\| T(u + h) - T(u + h) - T(u) - T'(u)h \|}{\| h \|}\notag\\
    &\quad + \| W \| \frac{\| T(u + h) - T(u) \|^2}{\| h \|} + \| R \| \| h \|.
\end{align}

In view of Fréchet differentiability of \(T\), the first term on the RHS of (\ref{eq:OPTIMALITY_4}) goes to zero as \(\|h\| \to 0\). Also, Fréchet differentiability of \(T\) implies that \(T\) is locally Lipschitz (refer Joshi and Bose \cite{joshi1985some}) and hence the second term also goes to zero as \(\|h\| \to 0\). This proves that RHS of (\ref{eq:OPTIMALITY_4}) goes to zero and hence LHS goes to zero as \(\|h\| \to 0\).

This gives
\[
\frac{1}{2} J'(u)h = \langle T'(u)h, WTu \rangle + \langle h, Ru \rangle, \quad u \in Z.
\]
\end{proof}

The following theorem gives the optimality system for (\ref{eq:Dev_1}). Here the superscript \(*\) corresponding to a given operator denotes its adjoint.

\begin{theorem}\label{theorem:Th_5_dev}
    Under Assumptions [V] the `optimality system' for (\ref{eq:Dev_1}) is given by
\[
x^* = KNx^* + Hu^*
\]
\[
u^* = -R^{-1}H^*[I - KG(x^*)]^{*-1}Wx^*.
\]
\end{theorem}

\begin{proof}
    Existence and uniqueness of the optimal pair \((u^*, x^*)\) is proved in Corollary \ref{cor:dev_1}. If \(u^*\) is an optimal control then it is necessary that \(J'(u^*) = 0\). Using Lemma \ref{lemma:dev_4}, we get
\[
\langle h, Ru^* \rangle + \langle T'(u^*)h, Wx^* \rangle = 0 \quad \text{for all } h \in Z \text{ where } x^* = Tu^*.
\]
Taking adjoint of the derivative of the system operator, denoted by \([T'(u^*)]^*\), we get
\[
\langle h, Ru^* \rangle + \langle h, [T'(u^*)]^*Wx^* \rangle = 0 \quad \text{for all } h \in Z.
\]

This implies that
\[
Ru^* = [T'(u^*)]^*Wx^*
\]
which gives
\[
u^* = -R^{-1}((I - KG(x^*))^{-1}H)^*Wx^*.
\]
That is,
\[
u^* = -R^{-1}H^*[I - KG(x^*)]^{*-1}Wx^*
\]
where \(x^*\) satisfies
\[
x^* = KNx^* + Hu^*.
\]

Thus the optimal pair \((u^*, x^*)\) satisfies the coupled operator equations
\[
x^* = KNx^* + Hu^*,
\]
\[
u^* = -R^{-1}H^*[I - KG(x^*)]^{*-1}Wx^*.
\]

\end{proof}

As special cases of this result, we shall derive optimality systems for parabolic equations involving linear and nonlinear operators in Section 6(XXXXXXX).

\begin{corollary}\label{cor:dev_2}
    Suppose that \(X = Z\) is a real Hilbert space. Assume that \(R = I = W\) and \(H = K\) in the above Theorem \ref{theorem:Th_5_dev}. Then the unique optimal pair \((u^*, x^*)\) satisfies the following optimality system:
\[
x^* = KNx^* + Ku^*,
\]
\[
u^* = K^*G^*(x^*)u^* - K^*x^*.
\]
\end{corollary}
\begin{proof}
    By Theorem 5.1, the optimal pair \((u^*, x^*)\) satisfies the following optimality system:
\[
x^* = KNx^* + Ku^*,
\]
\[
u^* = -K^*[I - KG(x^*)]^{*-1}x^*.
\]

That is,
\[
u^* = -K^*[I - G^*(x^*)K^*]^{-1}x^*.
\]
Therefore,
\[
[I - K^*G^*(x^*)]u^* = -[I - K^*G^*(x^*)]K^*[I - G^*(x^*)K^*]^{-1}x^*
\]
\[
= -K^*[I - G^*(x^*)K^*][I - G^*(x^*)K^*]^{-1}x^*.
\]

This implies
\[
-K^*x^* = u^* - K^*G^*(x^*)u^*,
\]
\[
u^* = K^*G^*(x^*)u^* - K^*x^*.
\]
 and hence the proof.
\end{proof}

\begin{corollary}\label{cor:dev_3}
    For the linear system, that is, when \(N = 0\), the optimality system is given by
\[
x^* = Hu^*,
\]
\[
u^* = -R^{-1}H^*Wx^*.
\]
\end{corollary}

\section{Optimality Results for Systems Governed by Differential Equations}

We now derive sufficient conditions for the existence of optimal control for the class of nonlinear differential equation (\ref{eq:P}) defined in Section 1(XXXXXX). For that, we make the following assumptions on the system components.

\textbf{Assumptions [VI]:}
\begin{itemize}
   \item[(a)]: There exists a positive constant \(\mu\) such that
    \[
    \langle -Ax, x \rangle \geq \mu \|x\|^2 \quad \text{for all } x \in D, t \in [t_0, t_1].
    \]
    \item[(b)]: The nonlinear function \(F : [t_0, t_1] \times Y \to Y\) satisfies a growth condition of the form
    \[
    \|F(t, x)\| \leq a(t) + b\|x\| \quad \forall t \in [t_0, t_1], \, \forall x \in Y
    \]
    where \(a \in L^2[t_0, t_1]\), \(b > 0\).  
    Further \(F\) is negative monotone, that is,
    \[
    \langle F(t, x) - F(t, y), x - y \rangle \leq 0 \quad \text{for all } x, y \in Y, t \in [t_0, t_1].
    \]
    \item[(c)]: For all \(t > s\), \(G(t, s)\) is a compact  operator and \(B\) is a bounded linear operator for all \(t \in [t_0, t_1]\).
\end{itemize}

\begin{lemma}\cite{jha2025approximate}\label{lemma:dev_5}
    Under Assumptions [VI], the system operator \(T\) corresponding to the system (\ref{eq:P}) is completely continuous.
\end{lemma}

\begin{remark}
     Suppose that the assumption [VI(c)] in the above theorem is replaced by  
(c') : $B(t)$ is a compact bounded linear operator from $U$ to $Y$ for all $t \in [t_0, t_1]$. Then it can be shown by using Lemma \ref{lemma:dev_2} that the system operator $T$ is weakly continuous.
\end{remark}

\begin{remark}
    By virtue of Lemma \ref{lemma:dev_1} and \ref{lemma:dev_2} we can give different sets of verifiable assumptions on $A$, $B$ and $F(t, x(t))$ which will guarantee the complete and weak continuity of $T$. For example, if $A$ is a closed linear operator and generator of an almost strong evolution operator, $B$ is a bounded linear operator and $F(t, x(t))$ is Lipschitz, then the weak continuity of $T$ can be verified by using Remark 4.2(XXXXXX).

We first consider the constrained case. Assume that $U_{ad}$ is a weakly compact subset of $U$.
\end{remark}

\textbf{Assumption [VII]:}  
\begin{itemize}
    \item[(a)]: $g : [t_0, t_1] \times U_{ad} \times Y \to \mathbb{R}$ is approximately lower semi continuous.
    \item[(b)]: For every $t \in [t_0, t_1]$, $(u, x) \to g(t, u, x)$ is lower semi continuous with respect to the weak topology in $U_{ad}$ and strong topology in $Y$.
    \item[(c)]: For all $(t, x) \in [t_0, t_1] \times Y$, $u \to g(t, u, x)$ is convex.
\end{itemize}

An easy application of Lemma \ref{lemma:dev_5} and Theorem 3.1(XXXXXXXXXX) gives us the following result regarding the existence of an optimal pair for the system (\ref{eq:P}).

\begin{theorem}\label{theorem:Th_6_dev}
     Under Assumptions [VI] and [VII] there exists an optimal pair $(u^*, x^*) \in U_{ad} \times X$ for the nonlinear evolution equation (\ref{eq:P}).
\end{theorem}

\begin{proof}
    Define 
\[
S = \{w \in L^2([t_0, t_1], U) : w(t) \in U_{ad} \text{ a.e.} \}.
\]
Since $U_{ad}$ is weakly compact we have that $S$ is also weakly compact in $L^2([t_0, t_1], U)$.

Let $\{u_n\}$ be a sequence with weak limit $u^*$ in $S$. Then by Lemma \ref{lemma:dev_5} the corresponding response $\{x_n\}$ converges strongly to $x^*$ where $x^*$ is the response of $u^*$. That is, $x_n(t) \to x^*(t)$ in $Y$ whenever $u_n(t) \to u^*(t)$ in $U_{ad}$. Following Papageorgiou \cite{papageorgiou1987existence}, the Assumptions [VII] implies that
\[
\lim_{n \to \infty} \int_{t_0}^{t_1} g(t, u_n(t), x_n(t)) dt \geq \int_{t_0}^{t_1} g(t, u^*(t), x^*(t)) dt
\]
whenever $u_n \to u^*$. 

That is, $\lim \inf J(u_n) \geq J(u^*)$ proving the weak lower semi-continuity of $J$ on the weakly compact set $S$. Now by Theorem 3.1(XXXXX), there exists an optimal pair $(u^*, x^*)$ for the system (\ref{eq:P}).
\end{proof}

For the unconstrained problem, we take $g(t, u, x)$ to be of the special form
\[
g(t, u, x) = \langle u(t), R_0 u(t) \rangle + \langle x(t), W_0 x(t) \rangle.
\]
Then the cost functional $J(u)$ assumes the form
\begin{equation}\label{eq:J(u)}
    J(u) = \Phi(u, x) = \int_{t_0}^{t_1} \langle u(s), R_0 u(s) \rangle + \langle x(s), W_0 x(s) \rangle ds
\end{equation}

where $W_0 : Y \to Y$, $R_0 : U \to U$ are bounded linear operators and $Y$, $U$ are real Hilbert spaces.\\

\textbf{Assumptions [VIII]:}
\begin{itemize}
    \item[(a)]: The operators $A$, $B$, and $F$ satisfy Assumptions [VI].
    \item[(b)]: $W_0$ is a bounded linear symmetric monotone operator and $R_0$ is a bounded linear symmetric strongly monotone operator, that is, there exists a constant $a > 0$ such that
    \[
    \langle u, R_0 u \rangle \geq a \|u\|^2 \quad \text{for all } u \in U.
    \]
\end{itemize}

\begin{theorem}\label{theorem:Th_7_dev}
    Under Assumptions [VIII], the system (\ref{eq:P}) has a unique optimal pair $(u^*, x^*)$ with respect to the cost functional (\ref{eq:J(u)}).
\end{theorem}

\begin{proof}
    Define $W : L^2([t_0, t_1], Y) \to L^2([t_0, t_1], Y)$ and $R : L^2([t_0, t_1], U) \to L^2([t_0, t_1], U)$ by $(W x)(t) = W_0 x(t)$ and $(R u)(t) = R_0 u(t)$.  
It follows easily that $R$ and $W$ are both bounded linear operators (refer Joshi and Bose \cite{joshi1985some}).

Further,
\[
\langle x, W x \rangle = \int_{t_0}^{t_1} \langle x(s), W_0 x(s) \rangle ds,
\]
\[
\langle u, R u \rangle = \int_{t_0}^{t_1} \langle u(s), R_0 u(s) \rangle ds.
\]

Thus (\ref{eq:J(u)}) can be written as
\[
J(u) = \langle u, R u \rangle + \langle x, W x \rangle_X.
\]

It is easy to verify that $R$ is a strongly monotone symmetric operator and $W$ is linear, monotone, and symmetric. Now the theorem follows by using Theorem \ref{theorem:Th_4_dev} and Corollary \ref{cor:dev_1}.
\end{proof}

\begin{remark}\label{remark:dev_1}
    If we assume that the nonlinear function $F(t, x)$ is continuously Fréchet differentiable with respect to $x$ for almost all $t \in [t_0, t_1]$ with $G(t, x) = \frac{\partial}{\partial x} F(t, x)$, then it follows that the operator $N$, as defined by (6.4)(iii)(XXXXXXXX) in the space $X = L^2([t_0, t_1], Y)$, is continuously Fréchet differentiable (refer \cite{joshi1985some}) with $[N(x)]u = Gu$, where $G : X \to X$ is defined by $(Gu)(t) = [G(t, x(t))]$.
\end{remark}

As a particular case to this, we have the following theorem regarding the optimality system for (\ref{eq:P}). We note that the system (\ref{eq:P}) is of the same type as considered by Seidman and Zhou \cite{seidman1982existence}. However, we impose monotonicity assumptions on $F$ in contrast to Lipschitz assumptions imposed by Seidman and Zhou \cite{seidman1982existence}. Also, we observe that we do not require Lipschitz continuity on the Fréchet derivative of $F$, whereas Seidman and Zhou require so. Moreover, our system (\ref{eq:P}) is non-autonomous. XXXXXXXXXXXXXXXXXXXXXXXXXXXXXXXXXXXXXXXXXXXX

\begin{theorem}\label{theorem:Th_8_dev}
     Suppose that $Y = U$ is a Hilbert space and $B = R_0 = W_0 = I$, and the operators $A$ and $F$ satisfy Assumptions [VIII(a)]. Further, assume that the nonlinear function $F(t, x(t))$ is continuously Fréchet differentiable with respect to $x$ for almost all $t \in [t_0, t_1]$ with $G(t, x) = \frac{\partial}{\partial x} F(t, x)$.

Suppose that
\[
\int_{t_0}^{t_1} \|G(t, x(t))\|^2 dt < \infty \quad \text{for all } x \in Y.
\]
Then the optimality system for (\ref{eq:P}) with cost functional (\ref{eq:J(u)}) is given by
\[
\dot{x}^* = Ax^*(t) + F(t, x^*(t)) + u^*(t),
\]
\[
\dot{u}^* = A^*u^*(t) + G^*(t, x^*(t))u^*(t) - x^*(t),
\]
with
\[
x^*(0) = x_0, \quad u^*(a) = 0.
\]
\end{theorem}

\begin{proof}
     The existence of an optimal pair $(u^*, x^*)$ follows from Theorem \ref{theorem:Th_8_dev}. Using the definitions of the operators $K$, $N$, and $H$ and Remark \ref{remark:dev_1}, it follows that $K'(x) = K$, $N'(x) = G(x)$, and $H'(u) = K$ for all $x$, $u \in X$. For $v \in X$, consider the operator equation
\[
u = KG(x)u + v,
\]
for a fixed $x \in X$.

Now using the definition of operators, we can write it as
\begin{equation}\label{eq:dev_last}
    u(t) = \int_{t_0}^{t} \Phi(t, s)G(s, x(s))u(s)ds + v(t)
\end{equation}

Since $\|\Phi(t, s)\| \leq M$ and $\int_{t_0}^{t_1} \|G(s, x(s))\|^2 ds < \infty$, we have
\[
\int_{t_0}^{t_1} \|\Phi(t, s)G(s, x(s))\|^2 ds < \infty.
\]

Hence, for each fixed $x$ and $v$, (\ref{eq:dev_last}) is a linear Volterra integral equation with $L^2$ kernel. Thus for each $v \in X$, (\ref{eq:dev_last}) has a unique solution in $X$. That is $[I - KG(x)]^{-1}$ exists and is linear. Moreover, this inverse is bounded. Hence by Theorem \ref{theorem:Th_5_dev}, it follows that the optimal pair $(u^*(t), x^*(t))$ satisfies the equations:
\[
x^*(t) = \Phi(t, t_0)x_0 + \int_{t_0}^t \Phi(t, \tau)F(\tau, x^*(\tau))d\tau + \int_{t_0}^t \Phi(t, \tau)u^*(\tau)d\tau
\]
\[
u^*(t) = \int_t^{t_1} \Phi^*(\tau, t)G^*(\tau, x^*(\tau))u^*(\tau)d\tau - \int_t^{t_1} \Phi^*(\tau, t)x^*(\tau)d\tau
\]
where $\Phi(t, \tau)$ and $\Phi^*(\tau, t)$ are the evolution operators generated by $A(t)$ and $A^*(t)$, respectively.

Differentiating with respect to $t$, we get
\[
\dot{x}^* = A(t)x^*(t) + F(t, x^*(t)) + u^*(t),
\]
\[
\dot{u}^* = A^*(t)u^*(t) + G^*(t, x^*(t))u^*(t) - x^*(t),
\]
with
\[
x^*(0) = x_0, \quad u^*(a) = 0.
\]
\end{proof}

\section{Application}

Let $\Omega$ be a bounded domain in $\mathbb{R}^d$ with smooth boundary $\partial \Omega$. For fixed $T > 0$, let $Q = (0,T) \times \Omega$ and $\Sigma = (0,T) \times \partial \Omega$. Let $A$ be a second order uniformly elliptic differential operator given by (1.2)XXXX. Set $X = L^2(\Omega)$, $V = H_0^1(\Omega)$, $D(A) = H^2(\Omega) \cap H_0^1(\Omega)$ and $D(B) = H^2(\Omega)$. Then the weak formulation of the problem (\ref{eq:P}) is given by

\begin{equation}\label{eq:App_1}
(u_{\alpha,t}, \phi) + A(u, \phi) =   (Bv, \phi)+ (f(t,u(t)),\phi) \quad \forall\, \phi \in V,\; t \in [0,T], 
\end{equation}
\begin{equation*}
u(0) = \sum_{k=1}^{m} c_k u(t_k),
\end{equation*}

where $A(\cdot,\cdot) : H_0^1(\Omega) \times H_0^1(\Omega) \rightarrow \mathbb{R}$, $f(t;\cdot,\cdot) : H_0^1(\Omega) \times H_0^1(\Omega) \rightarrow \mathbb{R}$  is satisfies the  Caratheodory Condition.

From the Jha and Chen theorem (see  \cite{jha2025approximate} and \cite{chen2020existence}) , Let $-A $
be the infinitesimal generator of an $\alpha$-order fractional compact and analytic solution operator $\{T_\alpha(t)\}_{t \geq 0}$. For $u(0) \in D(A)$, the unique mild solution for the system (\ref{eq:P}) is given by 
\begin{equation}\label{eq:App_2}
 u(t) = \int_0^a G(t, s) \big[Bv(s) + f(s, u(s))\big] \, ds, \quad t \in J.
\end{equation}
where $G(t,s)$  is defined in equation  (\ref{eq:chen_9}).

For final time $t = a$, we obtain
\begin{equation}
    u(a)=KNu+ Hv,
\end{equation}
\begin{equation}\label{eq:App_3}
   u(a)=KNu+ KBv
\end{equation}
where the operator $K,N$ and $H$ are defined as before in Section 2(XXXXXX). 

 Also, set $Y = L^2(0,T;U)$, the solution operator to $W : U \rightarrow Y$ be compact, and by  using the Assumptions [VI] and the condition of the theorem \ref{theorem:Th_7_dev}, and hence there exists a unique optimal pair $(u^*,x^*)$.
 
Let $\{\mathcal{T}_h\}$ be a family of regular triangulation of $\Omega$ with $0 < h < 1$. For $\mathfrak{K} \in \mathcal{T}_h$, set $h_{\mathfrak{K}} = \text{diam}(\mathfrak{K})$ and $h = \max\{h_{\mathfrak{K}}\}$.

Set
\[
X_h = \{x_h \in C^0(\bar{\Omega}) : x_h|_{\mathfrak{K}} \in P_1(\mathfrak{K}),\, \mathfrak{K} \in \mathcal{T}_h,\, x_h = 0 \text{ on } \partial \Omega \},
\]
where $P_1(\mathfrak{K})$ is the space of linear polynomials is on $\mathfrak{K}$. Let $X_h$ be the finite dimensional subspace of $U$ consisting of constant elements defined on the triangulation $\mathcal{T}_h$. Then, the semidiscrete Galerkin approximation of (\ref{eq:App_1}) is defined by the following: Find $u_h \in X_h$ such that
\begin{equation}\label{eq:App_4}
(u_{\alpha,h,t}(t), \chi) + A(u_h(t),\chi) =    (Bv_h, \chi)+ (f(t,u(t)),\chi)  \quad \forall\, \chi \in X_h,\; t \in [0,a], 
\end{equation}
\[
u_h(0) = \sum_{k=1}^{m} c_k u_h(t_k),,
\]
where, $u_h(t_k)$ is the approximation to $u(t_k)$ in $X_h$.

Let $P_h : X \rightarrow X_h$ is the $L^2$-projection and let $\{G_h(t,s)\}$ denote the finite element analogue of $G(t,s)$ defined by the semidiscrete equation (\ref{eq:App_4}) with $u_0 = 0$ and $f = 0$. This operator on $X_h$ may be defined as the semigroup generated by the discrete analog $A_h: X_h \rightarrow X_h $ of $A$, where
\[
(A_h x, \chi) = A(x, \chi) \quad \forall x, \chi \in V_h.
\]

Define the discrete analogue $f_h=f_h(t,u(t)): X_h \rightarrow X_h$ of $f=f(t,u(t))$ by

\[
(f_h(t,u(t))x,\chi) := f(t,u(t);x,\chi) \quad \forall x,\chi \in X_h,\; 0 \leq s \leq t \leq T.
\]

Now we write the semidiscrete problem (\ref{eq:App_4}) in an abstract form:
\begin{equation}\label{eq:App_5}
u_{\alpha ,h,t} + A_h u_h = P_h B v_h +f(t,u_h) =  P_h Bv_h+f_h u_h ,\quad t \in [0,T], 
\end{equation}

$$ u_h(0) = \sum_{k=1}^{m} c_k P_h u_h(t_k). $$

where $v_h(t) \in U_h$ and $u_h(t) \in X_h$. Using Duhamel’s principle, the solution $u_h$ of the semidiscrete problem (\ref{eq:App_5}) and  may be written as
\begin{align}\label{eq:App_6}
    u(t) &= \sum_{k=1}^m c_k T_{\alpha,h}(t) \mathcal{O} P_h\int_0^{t_k} S_{\alpha,h}(t_k - s) \big[Bv_h(s) + f_h(s, u_h(s))\big] \, ds 
\nonumber \\
    &\quad + \int_0^t S_{\alpha,h}(t - s) \big[Bv_h(s) + f_h(s, u_h(s))\big] \, ds, \quad t \in J.
\end{align}

\textbf{Full Discretization.} In order to discretize in the direction of $t$, we partition the $t$-axis in a uniform partition (not necessarily) by 
\[
0 = t_0 < t_1 < \cdots < t_n < \cdots < t_N \text{ with } t_N = T \text{ and set } I_n = (t_{n-1}, t_n].
\]
Let $W_N$ denote the set of scalar functions on $[0, T]$ which reduces to polynomial of degree $q - 1$ on each $I_n$. Let $Z_h^N = W_N \otimes X_h$ and $Y_h^N = W_N \otimes U_h$. In fact, $Z_h^N$ consists of functions defined on $[0, T]$ whose restrictions to $I_n$ is a polynomial of degree $\leq q - 1$ with its coefficients in $X_h$, that is, $Z_h^N$ consists of piecewise constant on each $I_n$. Similarly, $Y_h^N$ is defined for $q = 1$.

Let $\Delta t$ be the step size in time, $t_n = n\Delta t$, $n = 1, 2, \cdots, N$, where $N = T/\Delta t$. Let $\phi^n = \phi(t_n)$. For $\phi \in C([0, T])$, set

\[
{}^{C}D_t^\alpha \phi^n \approx \frac{1}{\Delta t^\alpha} \sum_{j=0}^{n} w_j^{(\alpha)} \phi^{n-j}, \quad 
w_j^{(\alpha)} = (-1)^j \binom{\alpha}{j}.
\]

\[
{}^{C}D_t^\alpha \phi^n \approx \frac{1}{\Gamma(2 - \alpha)} \sum_{j=0}^{n-1} \frac{\phi^{j+1} - \phi^j}{\Delta t^\alpha} \left[(n - j)^{1 - \alpha} - (n - j - 1)^{1 - \alpha}\right].
\]

\[
\delta_t^\alpha \phi^n \approx {}^{C}D_t^\alpha \phi(t_n).
\]

\textbf{Backward Euler Scheme:} For $\phi_h^N \in Z_h^N$ and $u_h^N \in U_h^N$ with $\phi_h^n|_{I_n} = \phi_h^n \in X_h$ and $u_h^N|_{I_n} = u_h^n \in U_h$ for $n = 1, 2, \dots, N$ and replacing the integral term by the left hand rectangular rule as

\begin{equation}\label{eq:App_7}
\int_0^{t_n} \varphi(s) ds \approx \Delta t \sum_{j=0}^{n-1} \varphi(t_j).
\end{equation}

Then the backward Euler scheme is given by

\begin{equation}\label{eq:App_8}
(\delta_t^\alpha u_h^n, \chi) + A(u_h^n, \chi) =(B v_h^n, \chi)+f_h(t;u_h^n,\chi), \quad \chi \in X_h.
\end{equation}

\[
u_h^0 = \sum_{k=1}^{m} c_k u_h(t_k) \text{ in } \Omega,
\]

Let $u_h^n = \sum_{i=1}^{N_h} \alpha_i^n \varphi_i$ and $v_h^n = \sum_{i=1}^{M_h} \beta_i^n \psi_i$, where $\{\varphi_1, \varphi_2, \dots, \varphi_{N_h} \}$ and $\{\psi_1, \psi_2, \dots, \psi_{M_h} \}$ are bases of $X_h$ and $U_h$, respectively. Note that the space $\phi_h^N$ can be identified as the space of matrices $\mathbb{M}$ of dimension $M_h \times (N + 1)$. Therefore, the  minimization problem in $\mathbb{M}$ and take $R(t,u)=I$:

Find $\beta^* \in \mathbb{M}$ such that
\begin{equation}\label{eq:App_9}
\Phi_\varepsilon(\beta^*) = \inf_{\beta \in \mathbb{M}} \left( \Phi_\varepsilon(\beta) = \frac{1}{2} \|\beta\|_M^2 + \frac{1}{2} \|u_h\|_{X_h}^2 \right)
\end{equation}

\begin{enumerate}[label=\textbf{Step \arabic*:}]
\item Start with initial guess $\beta^{(0)}$ and tolerance $\epsilon$. Solve the state system using the Caputo-L1 scheme to compute $\boldsymbol{y}_h^{(0)}$. Compute:
\[
\mathcal{G}^{(0)} = \nabla_\beta \Phi_\epsilon(\beta^{(0)}), \quad \mathcal{D}^{(0)} = -\mathcal{G}^{(0)}.
\]

\item Update the control via line search:
\[
\beta^{(i+1)} = \beta^{(i)} + \alpha^{(i)} \mathcal{D}^{(i)},
\]
where $\alpha^{(i)}$ minimizes:
\[
\Phi_\epsilon(\beta^{(i+1)} )=\Phi_\epsilon(\beta^{(i)} + \alpha^{(i)} \mathcal{D}^{(i)}).
\]

\item If $\| \beta^{(i+1)} - \beta^{(i)} \| \leq \epsilon_1$, set $\beta^{(i)} = \beta^{(i+1)}$ and go to Step 6. Otherwise, continue.

\item Compute:
\[
\mathcal{G}^{(i+1)} = \nabla_\beta \Phi_\epsilon(\beta^{(i+1)}), \quad
\zeta^{(i)} = \frac{ \langle \mathcal{G}^{(i+1)}, \mathcal{G}^{(i+1)} \rangle }{ \langle \mathcal{G}^{(i)}, \mathcal{G}^{(i)} \rangle },
\]
and update:
\[
\mathcal{D}^{(i+1)} = -\mathcal{G}^{(i+1)} + \zeta^{(i)} \mathcal{D}^{(i)}.
\]

\item Set $i := i + 1$ and return to Step 2.

\item Update $\epsilon := \epsilon + \delta$, with $\delta > 0$, and set $\beta^{(i)} = \beta^{(i+1)}$.

\item If $\| \beta^{(i+1)} - \beta^{(i)} \| \leq \epsilon_2$, then set $\beta^* = \beta^{(i)}$ and stop. Otherwise, return to Step 2.
\end{enumerate}

To find the optimal control, we use this minimizer $\beta^*$. By $u^* = \sum \beta_{ij}^* \phi_i \psi_j$, we get $u_h^* \in X_h$ as computed from the semidiscrete optimality condition (\ref{eq:App_8}). From the algorithm, it is seen that the method finds the optimal control $u_h^*$ for fixed values of $N, h$, and in the outer loop, we increment it to find the optimal control for the space-time continuous optimal control problem (1.8)(XXXXXXX).

\section{Numerical experiment}
In this section, we present a numerical experiment to illustrate the computation of the minimizer $u^*$ with the operator $B = I$, identity operator. We consider the following one-dimensional initial-boundary value problem
\begin{equation}\label{eq:App_10}
\frac{\partial^{\alpha} u(t,x)}{\partial t^\alpha} - \frac{\partial^2 u(t,x)}{\partial x^2} =  v(t,x) +f(t,x,u(t,x)), \quad \text{on } Q := (0,1) \times (0,1), 
\end{equation}
\[
y(0,x) = y_0(x), \quad y(t,0) = y(t,1) = 0 \quad x \in (0,1),\quad t \in [0,1].
\]

Set $T = 1$, $\Omega = (0,1) \subset \mathbb{R}^1$ with $B(t,s) = \exp(-(t-s)^2 - u(s))$, $y_0(x) = \sin(\pi x)$ and $g = \exp(-x^2)\sin(\pi x)$. Note that for the above problem, $B(t,s) = \exp(-(t-s)^2)$, so set of reachable states, since $\exp(-x^2)\sin(\pi x)$ is not an exact solution of the system (4.1). Any corresponding to the control function $u(t,x) = -\exp(-x^2)\pi\sin(\pi x)$ with $y(T,x) = \exp(-x^2)\sin(\pi x)$.

\begin{figure}[H]
    \centering
    \includegraphics[width=1.1\textwidth]{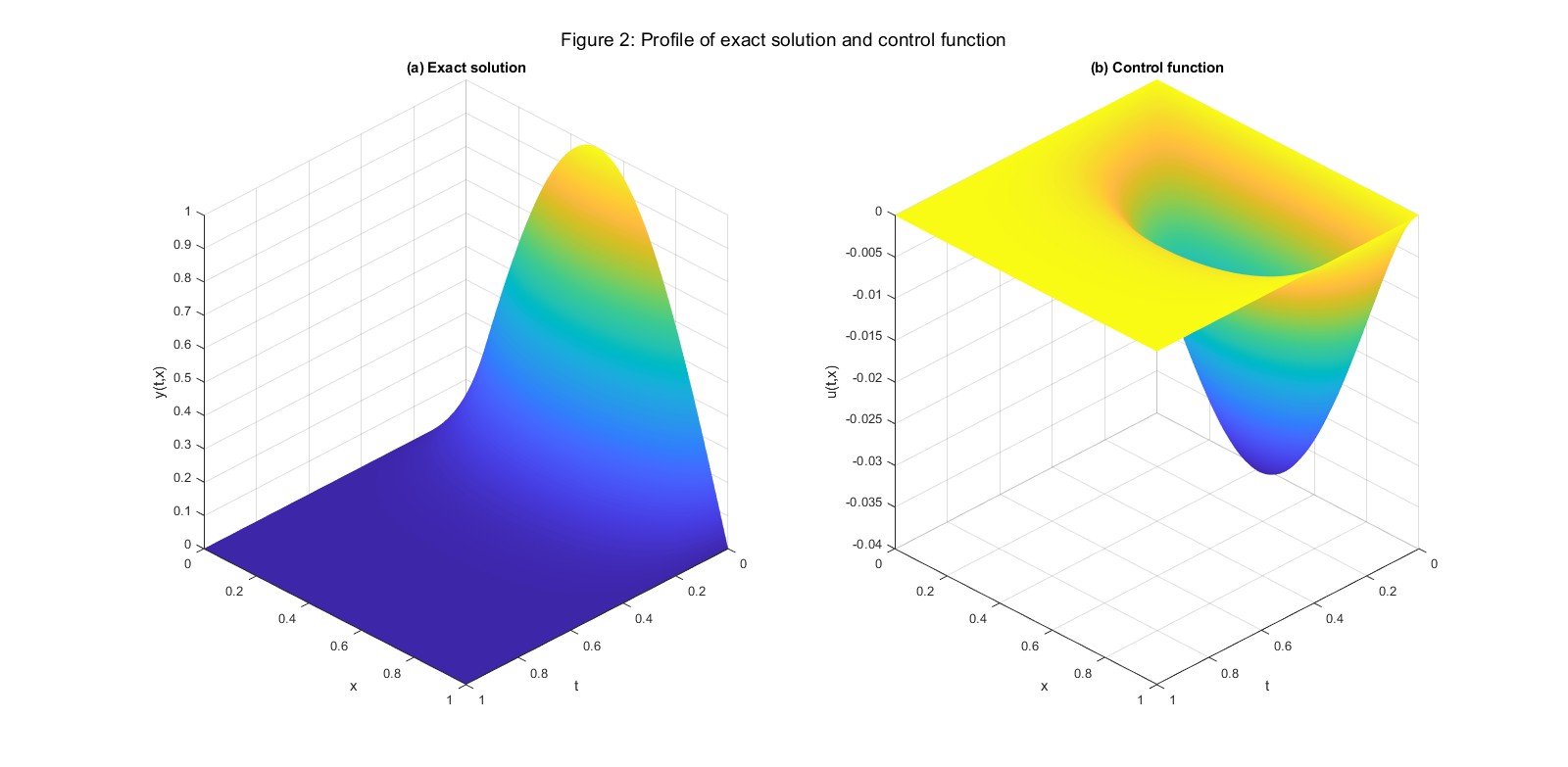}
    \caption{Profile of exact solution and control function}
    \label{fig:exact}
\end{figure}

\begin{figure}[H]
    \centering
    \includegraphics[width=1\textwidth]{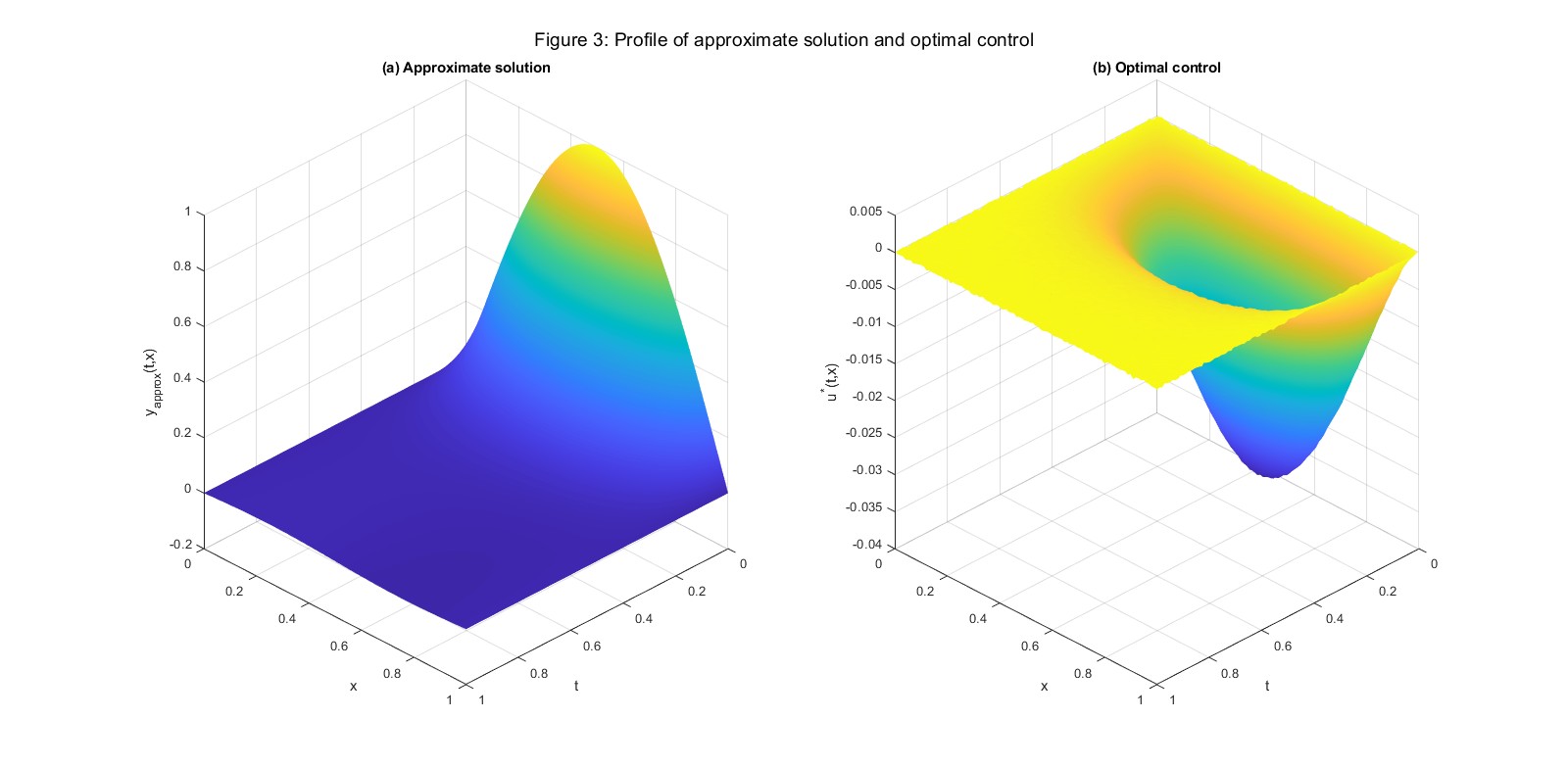}
    \caption{Profile of approximate solution and optimal control}
    \label{fig:approx}
\end{figure}

Here, we choose $\Delta t$, $h$ and $N = 1/\Delta t$. Using the MOA algorithm, we compute $u_n^*$, $n = 1, 2, \ldots, N$ and then plot the graph of numerical results for $N = 40$. In Fig. 1, we plot the graph of the approximated state at time $T = 1$ and the given final state $j = \exp(-x^2)\sin(\pi x)$ corresponding to the approximated optimal control $u^*$. The solution profile of the exact solution $y(t,x) = \exp(-t^2)\sin(\pi x)$ corresponding to the control function $u(t,x) = -\exp(-x^2)\pi\sin(\pi x)$ is shown in Fig. 2. Fig. 3 shows the solution profile of the optimal solution corresponding to the optimal control computed by using the MOA algorithm.

\bibliographystyle{unsrtnat} 
\bibliography{references} 

\end{document}